\documentclass{amsart}

\usepackage{amssymb,amsmath,amsthm,color}
\usepackage[all]{xy}
\numberwithin{equation}{section}
\newtheorem{theorem}{Theorem}

\newtheorem{corollary}[theorem]{Corollary}
\theoremstyle{definition}

\newtheorem{remark}[theorem]{Remark}

\newcommand{\Z}{\mathbb{Z}}
\newcommand{\N}{\mathbb{N}}

\newcommand{\calH}{\mathcal{H}}

\newcommand{\calL}{\mathcal{L}}

\newcommand{\calP}{\mathcal{P}}

\newcommand{\calT}{\mathcal{T}}
\newcommand{\calU}{\mathcal{U}}

\newcommand{\frake}{\mathfrak e}

\newcommand{\id}{\operatorname{id}}

\newcommand{\p}[2]{P_{#1, #2}(z)}

\makeatletter 
\makeatletter 
\let\@@pmod\pmod 
\DeclareRobustCommand{\pmod}{\@ifstar\@pmods\@@pmod} 
\def\@pmods#1{\mkern4mu({\operator@font mod}\mkern 6mu#1)}

\begin{document}

\title[Correction to: The Fourier expansion of  Hecke operators]{Correction and addendum to: The Fourier expansion of  Hecke operators for vector-valued modular forms}

\author{Oliver Stein}
\address{Fakult\"at f\"ur Informatik und Mathematik, Ostbayerische Technische Hochschule Regensburg\\Galgenbergstrasse 32\\93053 Regensburg\\Germany}
\email{oliver.stein@oth-regensburg.de}
\thanks{}
\thanks{}

\begin{abstract}
We correct a mistake in \cite{St} leading to erroneous formulas in Theorems 5.2 and 5.4. As an immediate corollary of a  formula in \cite{BCJ} we give a formula, which relates the Hecke operators $T(p^2)\circ T(p^{2l-2})$, $T(p^{2l})$ and $T(p^{2l-4})$ and comment on it.  
\end{abstract}

\maketitle

\section{Introduction}\label{eq:subsec_mistakes}
There is a slight error in the calculations which lead to the (unfortunately incorrect) formula (5.3) in Theorem 5.2 in  \cite{St} and subsequently to slightly erroneous formulas in Theorem 5.4.  I was made aware of this mistake by Bouchard et. al. in their paper \cite{BCJ}. I am grateful to them for pointing this out. All other statements are unaffected. In this note we correct these mistakes and give a formula for  algebraic relations the Hecke operators $T(p^{2l})$ satisfy.  We also add a remark regarding this formula. 

We use the same notation as in \cite{St}. In particular, for integers $a,b$ by $(a,b)$ we mean the greatest common divisor of $a$ and $b$.  Additionally, we adopt the some notation from \cite{BCJ}. In particular, for an integer $n$ we use for the scaled quadratic form the symbol $q_n$, that is
\begin{equation}
  q_n(\cdot) = nq(\cdot).
  \end{equation}
Many calculations of Theorem 5.2 and subsequently of Theorem 5.4 in \cite{St} are based on
the identity
\begin{equation}\label{eq:gauss_sum_identity}
  \sum_{v\in L/p^sL}e\left(\frac{1}{p^s}(q(v+\lambda))\right) = e(\frac{1}{p^s}q(\lambda))\sum_{v\in L/p^sL}e\left(\frac{1}{p^s}q(v)\right), 
\end{equation}
where $s\in \Z$ is positive integer and $\lambda\in \calL^{p^s}$. 
Unfortunately, this equation is incorrect:
The left-hand side of \eqref{eq:gauss_sum_identity} is independent of the representative of $\lambda$. Thus, passing from $\lambda$ to $\lambda + w$ for some $w\in L$ leaves the sum on the left unchanged.
However, the expression $e(\frac{1}{p^s}q(\lambda))$ is not independent of the chosen representative. In fact, we have
\begin{equation}\label{eq:quadratic_form_id_1}
\frac{1}{p^s}q(\lambda+w) = \frac{1}{p^s}q(\lambda) + \frac{1}{p^s}(\lambda,w)+\frac{1}{p^s}q(v).
\end{equation}
The latter summand of \eqref{eq:quadratic_form_id_1} is not in $\Z$ and therefore
\[
e(\frac{1}{p^s}q(\lambda+w))\not=e(\frac{1}{p^s}q(\lambda)).
\]

Later in the proof of Theorem 5.2 it is  claimed that the identity
\begin{equation}\label{eq:gauss_sum_id}
\sum_{v\in L/p^sL}e\left(\frac{t}{p^s}q(v+h\lambda - p^l\nu)\right) = e\left(\frac{t}{p^s}q(h\lambda-p^l\nu)\right)\sum_{v\in L/p^sL}e\left(\frac{t}{p^s}q(v)\right)
\end{equation}
holds. Here $\nu$ runs through $\calL$ and $h\lambda-p^l\nu$ is an element of $\calL^{p^m}$ for each $\nu$ ($m=(s,l)$), which implies that $\lambda\in \calL^{p^m}$. The first sum is part of a bigger expression involving a sum over $\calL$ parameterized by $\nu$, see \eqref{eq:action_beta_s}. The term  $e\left(\frac{t}{p^s}q(h\lambda-p^l\nu)\right)$ needs therefore to be independent of the choice of the representative of $\lambda$ and $\nu$. But, as pointed out above, this is not the case.
It follows that the sum over $v$ in \eqref{eq:action_beta_s} cannot be replaced by the right-hand side of \eqref{eq:gauss_sum_id}. 

\section{Correction}
\begin{theorem}[corrected Theorem 5.2]\label{thm:weil_representation}
  Let $p$ be an odd prime, $s,l$ positive integers with $s<2l$, $h\in (\Z/p^s\Z)^*$ and $D=\dim(L)$. Then
  \begin{equation}\label{eq:weil_represenstation}
    \frake_\lambda\mid \beta_{h,s} =
      \begin{cases}
        p^{-sD/2}\sum_{v\in L/p^sL}e\left(-\frac{h}{p^s}q(v+\lambda) \right)\frake_{p^{l-s}\lambda}, &  l\ge s,\\
       p^{-sD/2}\sum_{\substack{\mu\in \calL\\p^{s-l}\mu = \lambda}}\sum_{v\in L/p^lL}e\left(-\frac{hp^{s-l}}{p^l}q(v+\mu)\right)\frake_\mu, & l < s. 
        \end{cases}
  \end{equation}
\end{theorem}
\begin{proof}
We start with formula (5.9) in the proof of Theorem 5.2 and use subsequently equation (5.12) to obtain
\begin{equation}\label{eq:action_beta_s}
\begin{split}
  &  \frac{e(hrq(\lambda)}{\sqrt{|\calL|}\sqrt{|\calL(p^s)|}}\sum_{\rho\in \calL}\sum_{\nu\in \calL}e(-p^{2l-s}tq(\nu)+b(\nu,-\rho))\\
  &\times \sum_{\substack{\delta\in \calL(p^s)\\p^s\delta = h\lambda-p^l\nu \bmod{L}}}e(p^stq(\delta) + p^sb(\delta,-r\lambda))\frake_\rho \\
  &= \frac{e(hrq(\lambda)}{\sqrt{|\calL|}\sqrt{|\calL(p^s)|}}\sum_{\rho\in \calL}\sum_{\nu\in \calL}e(-p^{2l-s}tq(\nu)+b(\nu,-\rho))\\
  & \times e(b(h\lambda-p^l\nu,-r\lambda))\sum_{v\in L/p^sL}e\left(\frac{t}{p^s}q(v+h\lambda-p^l\nu\right).
\end{split}
\end{equation}

Note that $p^s$ and $h$ are related by the equation $rp^s-ht=1$. Therefore, $(|L/p^sL),h)=1$ and the map $v\mapsto hv$ is an isomorphism on $L/p^sL$. Thus
\begin{equation}\label{eq:gauss_sum_id_1}
  \begin{split}
  & \sum_{v\in L/p^sL}e\left(\frac{t}{p^s}q(v+h\lambda-p^l\nu)\right) \\
  &= \sum_{v\in L/p^sL}e\left(\frac{t}{p^s}q(hv+h\lambda-p^l\nu)\right)\\
  &=e(p^{2l-s}tq(\nu))\sum_{v\in L/p^sL}e\left(\frac{th^2}{p^s}q(v+\lambda) -\frac{th}{p^s}b(v+\lambda,p^l\nu)\right).
\end{split}
  \end{equation}
Note that the last sum over $v$ is independent of the choice of the representative of $\nu$ and $\lambda$ since
\begin{equation}\label{eq:gauss_sum_id3}
  \begin{split}
  &  \sum_{v\in L/p^sL}e\left(\frac{th^2}{p^s}q(v+\lambda) -\frac{th}{p^s}b(v+\lambda,p^l\nu)\right) \\
    & = e(-p^{2l-s}tq(\nu))\sum_{v\in L/p^sL}e\left(\frac{t}{p^s}q(v+h\lambda-p^l\nu)\right)
    \end{split}
\end{equation}
and both,  $e(-p^{2l-s}tq(\nu))$ and  $\sum_{v\in L/p^sL}e\left(\frac{t}{p^s}q(v+h\lambda-p^l\nu)\right)$ have this property.  
Using the relation $rp^s -ht =1$, the right-hand side of \eqref{eq:gauss_sum_id3} becomes
\begin{equation}\label{eq:gauss_sum_id_2}
  \begin{split}
    e(p^{2l-s}tq(\nu))e(hrq(\lambda))e(-rb(\lambda,p^l\nu))\sum_{v\in L/p^sL}e\left(-\frac{h}{p^s}q(v+\lambda) +\frac{1}{p^s}b(v+\lambda,p^l\nu)\right).
    \end{split}
\end{equation}

Replacing the sum over $v$ in \eqref{eq:action_beta_s} with \eqref{eq:gauss_sum_id_2} yields
\begin{equation}\label{eq:action_beta_s_1}
  \begin{split}
  &  \frac{1}{\sqrt{|\calL|}\sqrt{|\calL(p^s)|}}\sum_{\rho\in \calL}\sum_{\nu\in \calL}e(b(\nu,-\rho))\sum_{v\in L/p^sL}e\left(-\frac{h}{p^s}q(v+\lambda) +\frac{1}{p^s}b(v+\lambda,p^l\nu)\right)\frake_\rho \\
     &=\frac{1}{\sqrt{|\calL|}\sqrt{|\calL(p^s)|}}\sum_{\rho\in \calL}\sum_{v\in L/p^sL}e\left(-\frac{h}{p^s}q(v+\lambda)\right)\times\\
    &\sum_{\nu\in \calL}e\left(b(\nu,\frac{1}{p^s}(v+\lambda)-\rho)\right)\frake_\rho.
  \end{split}
  \end{equation}
Now, since $\nu \mapsto e\left(b(\nu,\frac{1}{p^s}(v+\lambda)-\rho)\right)$ is character of $\calL$, we have
\begin{equation}\label{eq:character_sum}
  \begin{split}
    \sum_{\nu\in \calL}e\left(b(\nu,\frac{1}{p^s}(v+\lambda)-\rho)\right) =
    \begin{cases}
      |\calL|, & \text{ if } \rho = p^l\left(\frac{1}{p^s}(v+\lambda)\right),\\
      0, & \text{ otherwise}. 
      \end{cases}
    \end{split}
\end{equation}
At this point we need to distinguish the cases $s\le l$ and $s> l$.

{\bf $s\le l$}:  In this case $p^l\left(\frac{1}{p^s}(v+\lambda)\right)$ is equal to $p^{l-s}(v+\lambda)\in \calL$ and we obtain for the last expression in \eqref{eq:action_beta_s_1}
\begin{equation}\label{eq:weil_beta_h_s_1}
  p^{-sD/2}\sum_{v\in L/p^sL}e\left(-\frac{h}{p^s}q(v+\lambda) \right)\frake_{p^{l-s}\lambda}.
  \end{equation}

{\bf $s>l$}:  As in \cite{St}, (5.9), we write the sum over $v$ on the right-hand side of \eqref{eq:action_beta_s_1} in the form
\begin{equation}\label{eq:gauss_sum_scaled_lattice}
\sum_{\substack{\delta\in \calL(p^s)\\p^s\delta=\lambda}}e(-hp^sq(\delta)).
\end{equation}
Since $\rho$ is an element of $\calL$, the sum in \eqref{eq:character_sum} is non-zero if and only if the multiplication of $\delta=\frac{1}{p^s}(v+\lambda)$ with $p^l$ yields an element of $\calL$. This is the case if and only if $\delta\in \calL(p^l)$. Thus, we may replace $\calL(p^s)$ in \eqref{eq:gauss_sum_scaled_lattice} with $\calL(p^l)$. Taking \eqref{eq:gauss_sum_scaled_lattice} and the thoughts before into account, we can replace the right-hand side of \eqref{eq:action_beta_s_1} with 
\begin{equation}\label{eq:weil_repr_identity}
  \begin{split}
  p^{-sD/2}\sum_{\substack{\delta\in \calL(p^l)\\p^s\delta = \lambda}}e(-hp^sq(\delta))\frake_{p^l\delta} &= p^{-sD/2}\sum_{\substack{\mu\in \calL\\p^{s-l}\mu = \lambda}}\sum_{\substack{\delta\in \calL(p^l)\\p^l\delta = \mu}}e(-hp^sq(\delta))\frake_\mu\\
  &= p^{-sD/2}\sum_{\substack{\mu\in \calL\\p^{s-l}\mu = \lambda}}\sum_{v\in L/p^lL}e\left(-\frac{hp^{s-l}}{p^l}q(v+\mu)\right)\frake_\mu. \\
  \end{split}
\end{equation}
\end{proof}

\begin{remark}
  Note that both  sums in \eqref{eq:weil_represenstation} are zero unless $\lambda \in \calL^{p^s}$ as remarked in \cite{St}, p. 245,  for the same type of sum.
  Compared to the original, but false formula, we don't have to add the separate symbol $\delta(\lambda,\cdot)$ to highlight this fact.

  Also note that the formulas \eqref{eq:weil_represenstation} coincide with the corresponding formulas in \cite{BCJ}. 
\end{remark}

The slightly different formulas for $\rho_L(\beta_{h,s})$ lead to slightly different formulas of the Fourier expansions of $T(p^{2l})$ compared to \cite{St}. To state the corrected theorem, we introduce the following notation. According to the decomposition (5.1) in \cite{St} we may write
\begin{equation}\label{eq:hecke_op_decomp}
f\mid_{k,L}T(p^{2l})(\tau) = g_\alpha(\tau) + \sum_{s=1}^{2l-1}\sum_{h\in (\Z/p^s\Z)^*}g_{\beta_{h,s}}(\tau) + \sum_{b\in \Z/p^{2l}\Z}g_b(\tau). 
\end{equation}
To lighten the formulas of the following theorem, we introduce some notation. 
For positive integers $s,l$ with $s<2l$ let $\lambda\in \calL^{p^{l-s}}, \lambda' \in \calL,  n\in \Z+q(\lambda)$ and  put
  \begin{align*}
    &  \mu(\lambda,\lambda') = \lambda/p^{l-s} +\lambda'\in \calL,\\
    & n(\lambda,\lambda') = \frac{n-p^{2(l-s)}q(\mu(\lambda,\lambda'))}{p^{2(l-s)}} + q(\mu(\lambda,\lambda'))) \in \Z+q(\mu(\lambda,\lambda')).
  \end{align*}
 Furthermore, attached to $\nu,\rho \in \calL$ and $m\in \Z+q(\nu)$, $r\in \Z+p^{2(s-l)}q(\rho)$ (assuming $s>l$) 
  we  define the  representation numbers modulo $a$
  \begin{equation}\label{eq:representation_no}
    \begin{split}
&      N_{\nu,m}(a) = |\left\{v\in L/aL\; |\; q(v+\nu) - m \equiv 0 \bmod{a}\right\}|, \\
 &     \widetilde{N}_{\rho,r}(a) = |\{v\in L/(a,p^l)L\; |\; q_{p^{2(s-l)}}(v+\rho) - r\equiv 0 \bmod{a}\}|
      \end{split}
  \end{equation}
and associated to them the sums
\begin{equation}\label{eq:sum_representation_no}
  \begin{split}
    &  G_{\nu,m}(s)= \sum_{a|p^s}\mu\left(\frac{p^s}{a}\right)a^{1-D} N_{\nu,m}(a),\\
    & \widetilde{G}_{\rho,r}(s) = \sum_{a\mid p^s}\mu\left(\frac{p^s}{a}\right)a(a,p^l)^{-D}\widetilde{N}_{\rho,r}(a). 
  \end{split}
\end{equation}
Note that the numbers in \eqref{eq:representation_no}  are well known. For example, they appear as a part of the Fourier expansion of vector valued Eisenstein series, cf. \cite{BK}. The number  $\widetilde{N}_{\rho,n}(a)$ can be interpreted as a variant of the representation number $N_{\lambda,n}(a)$ with respect to the scaled  quadratic form $q_{p^{2(s-l)}}$.

 In terms of these quantities we have
\begin{theorem}[corrected Theorem 5.4]\label{thm:Fourier_exp_hecke_op}
  Let $D=\dim(L)$, $p$ an odd prime,  $s,l$ positive integers with $s<2l$ and
  \[
     K_p=p^{2l(k-1)+s(D/2-k)}, \quad  \widetilde{K}_p = p^{2l(k-1)-sk+(2l-s)D/2}.
  \]
  Let $f\in M_{k,L}$  with Fourier expansion
\[
f(\tau) = \sum_{\lambda\in \calL}\sum_{\substack{ n\in \Z+q(\lambda)\\n\ge 0}}c(\lambda,n)e(n\tau)\frake_\lambda
\]
and
\[
\sum_{h\in (\Z/p^s\Z)^*}g_{\beta_{h,s}}(\tau) = \sum_{\lambda\in \calL}\sum_{\substack{ n\in \Z+q(\lambda)\\n\ge 0}}b_s(\lambda,n)e(n\tau)\frake_\lambda.
\]
Then for  $s=1,\dots,l$ 
\begin{equation}\label{eq:fourier_expansion_hecke_op_s_le_l}
  \begin{split}
    &    b_s(\lambda,n) =\\
  &  K_p\sum_{\substack{\lambda'\in \calL_{p^{l-s}}\\n-p^{2(l-s)}q(\mu(\lambda,\lambda'))\in p^{2(l-s)}\Z}}c\left(\mu(\lambda,\lambda'), n(\lambda,\lambda')\right)G_{\mu(\lambda, \lambda'), n(\lambda,\lambda')}(s).  
  \end{split}
\end{equation}
if $\lambda\in \calL^{p^{l-s}}$ and zero otherwise.
For $s=l+1,\dots, 2l-1$ 
   \begin{equation}\label{eq:fourier_expansion_hecke_op_l_le_s} 
    b_s(\lambda,n) =  \widetilde{K}_pc(p^{s-l}\lambda, p^{2(s-l)}n)\widetilde{G}_{\lambda,p^{2(s-l)}n}(s).
  \end{equation}
\end{theorem}
\begin{proof}
In view of section \ref{eq:subsec_mistakes}, we only need to adjust those parts of the Fourier expansion of $f\mid_{k,L}T(p^{2l})$ which involve the terms $\rho_L(\beta_{h,s})$, that is, the Fourier expansions of $\sum_{s=1}^{2l-1}\sum_{h\in (\Z/p^s\Z)^*}g_{\beta_{h,s}}(\tau)$. 
According to the formulas of $\rho_L(\beta_{h,s})$, we distinguish the cases $s \le l$ and $s>l$. 

For $s\le l$, by replacing the $f$ with its Fourier expansion and $\rho_L(\beta_{h,s})$ with  \eqref{eq:weil_beta_h_s_1}, we have
\begin{equation}\label{eq:fourier_exp}
  \begin{split}
    & \sum_{h\in (\Z/p^s\Z)^\times}\sum_{\lambda\in \calL}(f_\lambda\mid_k \beta_{h,s}) \rho_L^{-1}(\beta_{h,s})\frake_\lambda \\
    & =p^{k(1-s)-sD/2}\sum_{\lambda\in \calL} \sum_{\substack{ n\in \Z+q(\lambda)\\n\ge 0}}c(\lambda,n)e\left(\frac{np^{2l-s}\tau}{p^s}\right)\times \\
  & \sum_{v\in L/p^sL}\sum_{h\in (\Z/p^s\Z)^\times}e\left(\frac{h(q(v+\lambda) - n)}{p^s}\right)\frake_{p^{l-s}\lambda}.
    \end{split}
  \end{equation}
The sum over $(\Z/p^s\Z)^\times$ is a Ramanujan sum, which can be evaluated in terms of the Moebius function $\mu$. Using the same steps as in the proof of  \cite{BK}, Proposition 3, we obtain
\begin{equation}\label{eq:Ramanujan_sum}
\sum_{v\in L/p^sL}\sum_{h\in (\Z/p^s\Z)^\times}e\left(\frac{h(q(v+\lambda) - n)}{p^s}\right)=  p^{sD}\sum_{a|p^s}a^{1-D}\mu\left(\frac{p^s}{a}\right) N_{\lambda,n}(a),
\end{equation}
Inserting the right-hand side of \eqref{eq:Ramanujan_sum} into the Fourier expansion \eqref{eq:fourier_exp}, yields
\begin{align*}
  p^{k(1-s) + sD/2} \sum_{\lambda\in \calL} \sum_{\substack{ n\in \Z+q(\lambda)\\n\ge 0}}c(\lambda,n)G_{\lambda,n}(s)e\left(np^{2(l-s)}\tau\right)\frake_{p^{l-s}\lambda}.
  \end{align*}
At this point the proof in \cite{St} remains unchanged. We merely have to replace the Gauss sum $g(p^s, \chi_p^R, n-q(\lambda))$ with the sum $G_{\lambda,n}(s)$ and follow the subsequent steps to obtain the claimed result. 


$s> l$: We proceed in the same as way as before using the formula \eqref{eq:weil_repr_identity}:

\begin{equation}\label{eq:fourier_exp_beta_h_s}
  \begin{split}
  &  \sum_{h\in (\Z/p^s\Z)^\times}\sum_{\lambda\in \calL}(f_\lambda\mid_k \beta_{h,s}) \rho_L^{-1}(\beta_{h,s})\frake_\lambda \\
  &= \sum_{\lambda\in \calL^{p^{s-l}}}f_\lambda\left(\frac{p^{2l-s}\tau + h}{p^s}\right)\times\\
    &p^{k(1-s)-sD/2}\sum_{\lambda'\in \calL_{p^{s-l}}}\sum_{v\in L/p^lL}\sum_{h\in (\Z/p^s\Z)^\times}e\left(-h\frac{p^{s-l}}{p^l}q(v+\lambda/p^{s-l}+\lambda')\right)\frake_{\lambda/p^{s-l}+\lambda'}\\
    & =p^{k(1-s)-sD/2}\sum_{\rho\in \calL} \sum_{\substack{ n\in \Z+q(p^{s-l}\rho)\\n\ge 0}}c(p^{s-l}\rho,n)\times\\
    &\sum_{v\in L/p^lL}\sum_{h\in (\Z/p^s\Z)^\times}e\left(\frac{h(q(p^{s-l}v+p^{s-l}\rho)-n)}{p^s}\right)e\left(\frac{n\tau}{p^{2(s-l)}}\right)\frake_\rho.
  \end{split}
  \end{equation}
To evaluate the latter two sums, we perform the same steps as in the proof of Proposition 3,  \cite{BK}, with a slight modification including the fact that the first sum runs over $L/p^lL$ and not over $L/p^sL$.  We obtain
\begin{equation}\label{eq:Ramanujan_sum_id}
  \begin{split}
  &\sum_{v\in L/p^lL}\sum_{h\in (\Z/p^s\Z)^\times}e\left(\frac{h(q(p^{s-l}v+p^{s-l}\rho)-n)}{p^s}\right)\\
  &= p^{lD}\sum_{a\mid p^s}\mu\left(\frac{p^s}{a}\right)a(a,p^l)^{-D}\widetilde{N}_{\rho,n}(a).
  \end{split}
\end{equation}
It is well known (see e. g. \cite{McC}) that the sum over $(\Z/p^s\Z)^\times$   on the left-hand side of \eqref{eq:Ramanujan_sum_id} is zero unless
\[
\frac{p^s}{(p^s,q(p^{s-l}v+p^{s-l}\rho)-n)}
\]
is square-free. Equivalently, this sum is non-zero if and only if $p^{s-1}$ or $p^s$  divides $p^{2(s-l)}q(v+\rho) - n$. Since $s-1 \ge 2(s-l)$, we may conclude that in the Fourier expansion of $\sum_{h\in (\Z/p^s\Z)^*}g_{\beta_{h,s}}$ only those $n$ with $n-q_{p^{2(s-l)}}(\rho)\in p^{2(s-l)}\Z$ appear. This fact allows us to replace $n$ with $p^{2(s-l)}m$, where $m\in \Z+q(\rho)$.  We finally obtain for the Fourier expansion in this case
\begin{align*}
 & p^{k(1-s) + (2l-s)D/2}\sum_{\rho\in \calL}\sum_{\substack{m\in \Z+q(\rho)\\m\ge 0}}c(p^{s-l}\rho, p^{2(s-l)}m)\widetilde{G}_{\rho,p^{2(s-l)}m}(s)e(m\tau)\frake_\rho.
  \end{align*}
\end{proof}

\subsection{Algebraic relations of the Hecke operators $T(p^{2l})$}

Bouchard, Creutzig and Joshi (see \cite{BCJ}) introduced a Hecke operator $\calH_{n}$ on the space $M_{k,L}$ of vector valued modular forms transforming with the Weil representation, however on a completely different way than in \cite{BS}. Their approach is more general and allows them to deduce algebraic relations for the operators $\calH_{n}$ for {\it any} $n\in \N$, in particular for {\it all} $n\in \N$ with $(n,|\calL|)>1$. Moreover, Bouchard et. al. were able to prove that their Hecke operators coincide with the ones we defined in \cite{BS} and considered in \cite{St} and the present paper, that is, they proved $\calH_{p^{2l}} = T(p^{2l})$ {\it for all} primes $p$. Based on the equivalence of these two constructions we can carry over the algebraic relations of the operators $\calH_{p^{2l}}$ to the operators $T(p^{2l})$. 

Below we will briefly recall the definition of the Hecke operators $\calH_{n^2}$ and the algebraic relations they satisfy.

The definition of $\calH_{n^2}$ involves the two different operators $\calT_{n^2}$ and  $\calP_{n^2}$ (see Def. 3.8  and 4.2 in \cite{BCJ}). The statement of the algebraic relations depends on a third operator  $\calU_{n^2}$ (Def. 3.13 in \cite{BCJ}). 
The operators $\calU_{n^2}$ and $\calP_{n^2}$ can be interpreted as special instances of the operators $g\uparrow_{H}^A$ and $g\downarrow_{H}^A$ as for example studied in \cite{Br}, Chapter 3. 

For $n,k\in \N$, let  $M_{k,L}$ and $M_{k,L(n^2)}$ be the space of vector valued modular forms transforming with the Weil representation associated to $L$ and the scaled lattice $L(n^2)$, respectively.
Also, for any $\mu\in \calL(r)$ and $k,l\in \N$ with $kl=r$ we define
\[
\Delta_r(\mu,k) =
\begin{cases}
  1, & \text{ if } \mu\in \calL(l)\subset \calL(r),\\
  0, & \text{ otherwise}. 
  \end{cases}
\]
Then
\begin{enumerate}
\item[$i)$]
  \begin{equation}\label{def:T_n}
    \begin{split}
      &    \calT_{n^2}: M_{k,L}\rightarrow M_{k,L(n^2)},\quad F=\sum_{\lambda\in \calL}f_\lambda \frake_\lambda\mapsto \calT_{n^2}(F) \text{ with }\\
      &\calT_{n^2}(F)(\tau)= n^{2(k-1)}\sum_{\mu\in \calL(n^2)}\left(\sum_{\substack{r,s>0\\rs=n^2}}\frac{1}{s^k}\sum_{t=0}^{s-1}\Delta_{n^2}(\mu,r)e\left(-\frac{t}{r}q_{n^2}(\mu)\right)f_{s\mu}\left(\frac{r\tau+t}{s}\right)\right)\frake_\mu,
    \end{split}
  \end{equation}
\item[$ii)$]
  \begin{equation}\label{def:U_n}
    \begin{split}
     & \calU_{n^2}: M_{k,L}\rightarrow M_{k,L(n^2)},\quad F=\sum_{\lambda\in \calL}f_\lambda\frake_\lambda\mapsto \calU_{n^2}(F) \text{ with }\\
&      \calU_{n^2}(F)(\tau)=\sum_{\mu\in \calL(n^2)}\Delta(n,\mu)f_{n\mu}(\tau)\frake_\mu,
      \end{split}
  \end{equation}
\item[$iii)$]
  \begin{equation}\label{def:P_n}
    \begin{split}
      & \calP_{n^2}: M_{k,L(n^2)}\rightarrow M_{k,L},\quad F=\sum_{\mu\in \calL(n^2)}f_\mu \frake_\mu\mapsto \calP_{n^2}(F) \text{ with }\\
      & \calP_{n^2}(F)(\tau)=\sum_{\lambda\in \calL}\left(\sum_{\substack{\mu\in \calL(n)\\n\mu = \lambda}}f_\mu(\tau)\right)\frake_\lambda \text{ and }
    \end{split}
    \end{equation}
\item[$iv)$]
  \begin{equation}\label{def:H_n}
    \calH_{n^2}: M_{k,L}\rightarrow M_{k,L},\quad \calH_{n^2} = \calP_{n^2}\circ \calT_{n^2}.
    \end{equation}
\end{enumerate}

It is proved in \cite{BCJ} that $\calP_{n^2}$ is the left inverse operator to $\calU_{n^2}$, that is
\begin{equation}\label{eq:left_inverse}
  \calP_{n^2}\circ\calU_{n^2} = \id.
  \end{equation}
However, the relation $(\calU_{n^2}\circ \calP_{n^2})(F)= F$ is only valid if $F$ supported on $\calL(n)$ (i. e. $f_\lambda = 0$ if $\lambda\notin \calL(n)$).
As already mentioned, Bouchard et. al compared their Hecke operator with the Hecke operator $T(p^{2l})$ constructed in \cite{BS} and obtained
\begin{equation}\label{eq:comparision_hecke_op}
  T(p^{2l}) = \calH_{p^{2l}}
  \end{equation}
for all primes $p$ and all $l\in \N$.  
The following theorem summarizes the algebraic relations the Hecke operators $\calH_{n^2}$ satisfy.
\begin{theorem}[\cite{BCJ}, Theorem 4.12]
  \begin{enumerate}
    \item[$i)$]
      For $m, n\in \N$ with $(m,n)=1$ we have
      \[
      \calH_{m^2} \circ \calH_{n^2} = \calH_{m^2n^2}.
      \]
    \item[$ii)$]
      For any prime $p$ and any $l\in \N, \; l\ge 2$ the relation
      \begin{equation}\label{eq:rel_hecke_op}
        \calH_{p^{2l}} = \calP_{p^{2l-l}}\circ \calH_{p^2}\circ \calH_{p^{2l-2}}\circ\calU_{p^{2l-2}} - p^{k-1}\calH_{p^{2l-2}}-p^{2(k-1)}\calH_{p^{2l-4}} 
      \end{equation}
      holds.
      \end{enumerate}
\end{theorem}

The identities \eqref{eq:comparision_hecke_op} and \eqref{eq:rel_hecke_op} immediately yield to the corresponding relation for the operators $T(p^{2l})$ as a corollary.
\begin{corollary}
  For any prime $p$ and any $l\in \N$ the relation
  \begin{equation}\label{eq:hecke_operator_rel_bs}
    T(p^{2l}) = \calP_{p^{2l-2}}\circ T(p^2) \circ T(p^{2l-2})\circ \calU_{p^{2l-2}} - p^{k-1}T(p^{2l-2}) - p^{2(k-1)}T(p^{2l-4})
  \end{equation}
  holds.
\end{corollary}

\begin{remark}
  The formula \eqref{eq:hecke_operator_rel_bs} is not as definitive as the corresponding formula for the classical scalar valued Hecke operators due to the appearance of the operators $\calP_{p^{2l-2}}$ and $\calU_{p^{2l-2}}$. For example, it is unclear to me how to derive a recursion formula for the eigenvalues of a common Hecke eigenform from \eqref{eq:hecke_operator_rel_bs} (if it is even possible). Also, it does not provide the basis for the proof that the Hecke operators are commutative. 
 It would be desirable to have a formula expressing $T(p^2)\circ T(p^{2l-2})$ (without the operators $\calU_{p^{2l-2}}$ and $\calP_{p^{2l-2}}$) in terms of the Hecke operators $T(p^{2l})$ and $T(p^{2l-4})$ for all primes $p$. However, such a formula does seem not to exist. Even under quite strong restrictions on the Jordan block decomposition of the discriminant form $\calL$ (implying in particular  that $\calL$ is anisotropic) a formula of this type cannot be deduced. It would be interesting to know under what conditions this is possible.  
\end{remark}

\end{document}